\documentclass[12pt]{amsart}
\usepackage{amsmath}
\usepackage{amssymb}
\usepackage{amsthm}
\usepackage{epsfig,graphicx}
\usepackage[arrow,matrix,tips,2cell]{xy}
\newcommand{\ignore}[1]{\relax}

\newcommand{\Def}{\operatorname{Def}}
\numberwithin{equation}{section}

\newcommand{\C}{\mathbb C}
\newcommand{\R}{\mathbb R}
\newcommand{\Z}{\mathbb Z}

\newcommand{\val}{\operatorname{val}}

\newcommand{\deff}{\operatorname{def}}

\newtheorem{lem}{Lemma}[section]
\newtheorem{claim}{Claim}

\newtheorem{cor}[lem]{Corollary}

\newtheorem{theorem}{Theorem}
\newtheorem{corollary}[theorem]{Corollary}

\newtheorem{prop}[lem]{Proposition}
\newtheorem{add}[lem]{Addendum}

\theoremstyle{definition}
\newtheorem{defn}[lem]{Definition}

\newtheorem{exa}[lem]{Example}

\theoremstyle{remark}
\newtheorem{rmk}[lem]{Remark}
\newtheorem{rem}[lem]{Remark}

\newcommand{\tordva}{(\C^\times)^{2}}

\newcommand{\rtordva}{(\R^\times)^{2}}

\newcommand{\ev}{\operatorname{ev}}

\newcommand{\conj}{\operatorname{conj}}

\newcommand{\dd}{\partial}

\newcommand{\cp}{{\mathbb C}{\mathbb P}}
\newcommand{\rp}{{\mathbb R}{\mathbb P}}

\newcommand{\pp}{{\mathbb P}}

\renewcommand{\setminus}{\smallsetminus}

\newcommand{\MM}{{\mathcal M}}

\newcommand{\Area}{\operatorname{Area}}

\ignore{
\newtheorem{theorem}{Theorem}

\newenvironment{proof}[1][Proof]{\noindent\textbf{#1.} }{\ \rule{0.5em}{0.5em}}
}

\newcommand{\X}{\mathcal X}
\newcommand{\Y}{\mathcal Y}
\newcommand{\simple}{\operatorname{simple}}

\renewcommand{\S}{\mathcal S}
\newcommand{\ctor}{(\C^\times)^2}
\usepackage{cite}

\begin{document}

\title{On Block-G\"ottsche multiplicities for planar tropical curves.}
\author{Ilia Itenberg and Grigory Mikhalkin}
\address{Universit\'e Pierre et Marie Curie and
Institut Universitaire de France\\
Institut de Math\'ematiques de Jussieu\\
4 place Jussieu\\
75005 Paris, France}
\email{itenberg@math.jussieu.fr}
\address{Universit\'e de Gen\`eve\\
Math\'ematiques, villa Battelle\\
7, route de Drize\\
1227 Carouge, Switzerland}
\email{grigory.mikhalkin@unige.ch}
\begin{abstract}
We prove invariance for the number of planar tropical curves enhanced with
polynomial multiplicities recently proposed by Florian Block and Lothar G\"ottsche.
This invariance has a number of implications in tropical enumerative geometry.
\end{abstract}
\thanks{Research is supported in part by
ANR-09-BLAN-0039-01 grant of {\it Agence Nationale de la
Recherche}
(I.I.),
by the TROPGEO project of the European Research Council and
by the Swiss National Science Foundation  grants 125070 and 126817 (G.M.).
}

\maketitle

\section{Introduction}

\subsection{Some motivations}
One of the most classical problem in enumerative geometry is computing the number
of curves of given degree $d>0$ and genus $g\ge 0$ that pass through the appropriate number (equal to $3d-1+g$)
of generic points in the projective plane $\pp^2$.
This problem admits more than one way for interpretation. 
The easiest and the most well-studied interpretation is provided by the framework of complex geometry.
If we take a generic configuration of $3d-1+g$ points in $\cp^2$ the number of curves will only
depend on $d$ and $g$ and not on the choice of points as long as this choice is generic.
E.g. for $d=3$ and $g=0$ we always have $12$ such curves. 
For any given $g$ and $d$ the number   can be computed
e.g. with the help of the recursive relations of Caporaso-Harris \cite{CH}.

In this paper we are interested in setting up rather than solving plane enumerative problems.
In the world of complex geometry such set up is tautological: all relevant complex curves
are treated equally and each contributes 1 to the number we are looking for.
(Note that in this case all these complex curves 
are immersed and have only simple nodes as their self-intersection points.)

A somewhat less well-studied problem appears in the framework of real geometry. For
the same $d$ and $g$ but different choices of generic configurations of
$3d-1+g$ points the corresponding numbers of real curves can be different.
E.g. for $d=3$ and $g=0$ we may have $8$, $10$ or $12$ curves
depending on the choice of points (see~\cite{DKh}). 
It was suggested by Jean-Yves Welschinger \cite{W} 
to treat real curves differently for enumeration,
so that some real curves are counted with multiplicity $+1$ and some with multiplicity $-1$.
He has shown that the result is invariant on the choice of generic points if $g=0$.
E.g. for the $d=3$, $g=0$ case we always have 8 real curves counted with the Welschinger multiplicity.
This number may appear as 8 positive curves,
1 negative and 9 positive curves, or 2 negative and 10 positive curves.

Tropical enumerative geometry encorporates   features of both, real and complex geometry.
If we fix $3d-1+g$ generic points in the tropical projective plane then the corresponding number
of tropical curves of degree $d$ and genus $g$ can be different.
Nevertheless, tropical curves may also be prescribed multiplicities in such a way that
the resulting number is invariant.

So far two such recipes were known 
(see \cite{Mik05}): one recovering the number of curves 
for the complex problem and one recovering the number of curves for the real problem
enhanced with multiplicities corresponding to the Welschinger numbers.
Note that the real problem is only well-defined (and thus invariant) for the case of $g=0$, but the corresponding tropical
real problem is well-defined for arbitrary $g$, see \cite{IKS}.

Recently, a new type of multiplicity for tropical curves were proposed by Florian Block and Lothar G\"ottsche, \cite{BG}. 
These multiplicities are symmetric Laurent polynomials in one variable with positive integer coefficients.
According to the authors of this paper, which should appear soon, 
their motivation came from 
a Caporaso-Harris type calculation of 
the refined Severi degrees 
(introduced 
by G\"ottsche in connection with~\cite{KST})    
that interpolate
between the numbers of complex and real curves, see~\cite{GS}.   
Accordingly, their multiplicity for tropical curves interpolate between the complex and real multiplicities for tropical
curves: the value of the polynomial at 1 is the complex multiplicity while the value at $-1$ is the real multiplicity.

In this paper we show that the Block-G\"ottsche multiplicity is invariant of the choice of generic tropical configuration
of points and thus provides a new way for enumeration of curves in the tropical plane,
not unlike quantizing the usual enumeration of curves by integer numbers. E.g. if $d=3$ and $g=0$ then
the corresponding number is $y+10+y^{-1}$ that can come from eight curves with multiplicity 1 and
one curve of multiplicity $y+2+y^{-1}$, but also may come from nine curves with multiplicity 1 and one
curve of multiplicity $y+1+y^{-1}$.
The polynomial number of $y+10+y^{-1}$ curves can be thought of as 12 curves from complex enumeration,
but now this number decomposes according to different states: 10 ``curves" are in the  ground state, 
one ``curve" is excited in a $y$-state, while one ``curve" is excited in a $y^{-1}$-state. Here we use
quotation marks for curves as several of such virtual ``curves" correspond to the same tropical curve
(e.g. we have one tropical curve of multiplicity $y+1+y^{-1}$, but it correspond to three virtual ``curves" in different states).

Our considerations are not limited by curves in the projective planes and include enumeration in all toric
surfaces. As the configuration of tropical points is assumed to be generic, we may restrict our attention
to $\R^2$ (a tropical counterpart of $(\C^\times)^2$) which is dense in any tropical toric surface.
The corresponding toric degree is then given by a collection of integer vectors whose sum is zero.

\subsection{Tropical curves immersed in the plane}
A closed irreducible tropical curve $\bar C$ (cf. \cite{Mik06}, \cite{MZ} et al.) is a connected finite graph without 2-valent vertices
whose edges are enhanced with lengths.
The length of any edge which is not adjacent to a $1$-valent vertex
is a positive real number.
Any edge adjacent to a $1$-valent vertex is required to
have infinite length. Denote the set of 1-valent vertices of $\bar C$ with $\dd\bar C$. The lengths of the edges
induce a complete inner metric on the complement
\begin{equation}\label{opencurve}
C=\bar C\setminus\dd \bar C.
\end{equation}
A metric space $C$ is called an {\em open minimal tropical curve} if it can be presented by $\eqref{opencurve}$
for some closed irreducible tropical curve $\bar C$.

The number $\dim H_1(C; \R)$ of independent cycles
in $C$ is called
the {\em genus} of the curve $C$.
\begin{defn}[cf.  \cite{Mik05}]\label{def1.1}
An {\em immersed planar tropical curve}
is a smooth map $h:C\to\R^2$ (in the sense that it is continuous map whose restriction to any open edge is
a smooth map between differentiable manifolds), subject to the following properties.
\begin{itemize}
\item The map $h$ is a topological immersion.
\item For every unit vector $u\in T_y(C)$, where $y$ is inside an edge $E\subset C$, we have $(dh)_y(u)\in\Z^2$.
By smoothness, the image $(dh)_y(u)$ must be constant on the whole edge $E$ as long as we enhance
$E$ with an orientation to specify the direction of the unit vector. We denote $(dh)_y(u)$ with $u_h(E)$.
The GCD of the (integer) coordinates of $u_h(E)$ is called the {\em weight} $w_h(E)$ of the edge $E$.
\item For every vertex $v\in C$ we have $\sum\limits_E u_h(E)=0$,
where the sum is taken over all edges adjacent to $v$ and oriented away from $v$.
This condition is known as the {\em balancing condition}.
\end{itemize}
Recall that a continuous map is called {\em proper} if the inverse image of any compact is compact.
A proper immersed
tropical curve $h: C \to \R^2$ is called {\em simple} (see \cite{Mik05}) if it is $3$-valent,
the self-intersection points of $h$ are disjoint from vertices,
and the inverse image under $h$ of any self-intersection point
consists of two points of $C$.
\end{defn}

\begin{rmk}
Definition of tropical morphism which is not required
to be an immersion, or to spaces other than $\R^2$,
requires additional conditions which we do not treat here as we do not need them.
\end{rmk}

By Corollary 2.24 of \cite{Mik05} any simple tropical curve $h:C\to\R^2$
locally varies in a $(\varkappa+g-1)$-dimensional
affine space $\Def(h)$,
where $g$ is the genus of $C$ and $\varkappa$ is the number
of infinite edges of $C$.
This space has natural
coordinates once we choose a vertex $v\in C$. The two of those coordinates are given by $h(v)\in\R^2$.
The lengths of all closed edges of $C$ give $\varkappa+3g-3$ coordinates (since $h$ is simple the curve $C$ is 3-valent).
Then we have $2g$ linear relations
(defined over $\Z$) among these lengths
as each cycle of $h(C)$ must close up in $\R^2$.
By Proposition 2.23 of \cite{Mik05} these relations are independent.

Thus the space $\Def(h)$ is an open set in a $(\varkappa+g-1)$-dimensional
affine subspace $U\subset \R^{2+\varkappa+3g-3}$. The slope of this affine subspace is integer in the sense
that there exist $(\varkappa+g-1)$ linearly independent  vectors in $\Z^{2+\varkappa+3g-3}$ parallel to $U$.
This enhances the tangent space $T_h(U)$ to $\Def(h)$ at $h$ with integer lattice and
hence
a volume element (defined up to sign).

\subsection{Lattice polygons and points in general position}
Let $\Delta$ be a finite collection of non-zero vectors
with integer coordinates in $\R^2$ such that
the vectors
of $\Delta$ generate $\R^2$ and the sum
of these vectors is equal to $0$. We call such a collection
{\em balanced}.
The balanced collection $\Delta$ defines a {\em lattice polygon}
(that is, a convex polygon with integer vertices
and non-empty interior)
$\Delta^* \subset (\R^2)^*$
in the dual vector space $(\R^2)^*$ to $\R^2$:
each side $s$ of $\Delta^*$ is orthogonal to
a certain vector $v \in \Delta$ so that
$v$ is an outward normal to $\Delta^*$
(we say that such a vector $v$ is {\em dual} to $s$);
the {\it integer length} $\#(s \cap \ \Z^2) - 1$
of the side $s$ is equal to
the GCD of the two coordinates
of the sum of all the vectors in $\Delta$
which are dual to $s$.
The collection $\Delta$ defines a lattice polygon $\Delta^*$
uniquely up to translation.
Denote by $\varkappa(\Delta)$ the number of vectors in $\Delta$,
and denote by $\varkappa(\Delta^*)$
the {\em perimeter} $\#(\partial\Delta^* \cap \ \Z^2)$
of $\Delta^*$. Clearly, we have $\varkappa(\Delta)\le \varkappa(\Delta^*)$.

In general, if a lattice polygon $\Delta^*$ is fixed, the collection
$\Delta$ cannot be restored uniquely. However, if we assume
that all the vectors of $\Delta$ are {\em primitives}
(that is, the GCD of the coordinates of each vector is $1$ or, alternatively $\varkappa(\Delta)=\varkappa(\Delta^*)$),
then $\Delta^*$ defines $\Delta$ in a unique way.
A balanced collection $\Delta \subset \R^2$ is called {\em primitive}
if all the vectors of $\Delta$ are primitive.


We say that an immersed planar tropical curve $h:C\to\R^2$
is {\em of degree} $\Delta$ if
the multiset $\{u_h(E)\}$, where $E$ runs
over the unbounded edges $E \subset C$ oriented towards infinity,
coincides with $\Delta$.
Denote with $\MM_{g,\Delta}^{\simple}$ the space
of all simple tropical curves of degree $\Delta$
and genus $g$. As we saw, it is a disjoint union
of open convex sets in $\R^{\varkappa(\Delta) + g - 1}$ enhanced
with the a canonical choice of the integer lattice in its tangent space.

Recall (cf. Definition 4.7 of \cite{Mik05}) that
a configuration
$$\X=\{p_1, \dots, p_k\}\subset\R^2$$
is called
{\em generic} if for any
balanced collection $\tilde\Delta \subset \R^2$
and any non-negative integer number $\tilde g$
the following conditions hold.
\begin{itemize}
\item If $\varkappa(\tilde\Delta) + \tilde{g} - 1 = k$, then
any immersed tropical curve of genus $\tilde{g}$ and degree
$\tilde\Delta$ passing through $\X$ is simple
and its vertices are disjoint from $\X$.
The number of such curves is finite.
\item If $\varkappa(\tilde\Delta)+\tilde{g}-1<k$ there are
no immersed tropical curve of genus $\tilde{g}$ and degree
$\tilde\Delta$ passing through $\X$.
\end{itemize}
Proposition 4.11 of \cite{Mik05} ensures
that the set of generic configurations of
$k$ points in $\R^2$ are open and
everywhere dense in the space
of all configurations of $k$ points in $\R^2$.

\subsection{Tropical enumeration of real and complex curves}
Let us fix a
primitive balanced collection $\Delta\subset\R^2$
and an integer number $g\ge 0$.
For any generic configuration $\X=\{p_1,\dots,p_{k}\}\subset\R^2$ of
$k = \varkappa(\Delta) + g - 1$
points,
denote with $\S(g,\Delta,\X)$
the set of all curves of genus $g$ and degree $\Delta$
which pass through $\X$.

For any generic choice of $\X$ the set  $\S(g,\Delta,\X)$ is finite. Nevertheless, it might contain
different number of elements. Example 4.14 of \cite{Mik05} produces two choices $\X_1$ and $\X_2$ for
a generic configuration of three points in $\R^2$ such that
$\#(\S(0,\Delta,\X_1))=3$ but $\#(\S(0,\Delta,\X_2))=2$ for
the primitive balanced collection
$\Delta = \{(-1, 0), (0, -1), (2, -1), (-1, 2)\}$.

One can associate multiplicities $\mu$ to simple planar tropical curves so that
\begin{equation}\label{sum-multi}
\sum\limits_{h\in\S(g,\Delta,\X)}\mu(h)
\end{equation}
depends only on $g$ and $\Delta$ and {\em not} on the choice of a generic
configuration of $\varkappa(\Delta) + g - 1$ points $\X\subset\R^2$.

Previously there were known two ways to introduce such multiplicity:
the {\em complex multiplicity $\mu_\C$} and the {\em real multiplicity $\mu_\R$}
that will be defined in the next section.
The first multiplicity is a positive integer number while the second one is an integer number
which can be both positive and negative as well as zero.

These multiplicities were introduced in \cite{Mik05}. It was shown there
that the expression \eqref{sum-multi} for $\mu_\C$ adds up
to the number of complex curves of genus $g$
which are defined by polynomial with the Newton
polygon $\Delta^*$ and pass
through a generic configuration of $\varkappa(\Delta) + g - 1$ points in $\ctor$.
In the complex world this number clearly does not depend on the choice of the generic
configuration. If $\Delta^*$ is a triangle
with vertices $(0,0)$, $(d,0)$ and $(0,d)$, this
number coincides with the number of projective curves of genus $g$ and degree $d$
through $3d+g-1$ points (also known as one of the {\em Gromov-Witten numbers} of $\cp^2$,
cf. \cite{KM}).

For the real multiplicity we have a less well-studied situation.
Wel\-schinger \cite{W} proposed
to prescribe signs (multiplicities $\pm 1$) 
to rational real algebraic curves in $\rp^2$ (as well as
in other real Del Pezzo
surfaces).
A generic immersed algebraic curve $\R C$ in $\rp^2$ is {\em nodal}.
This means that the only singularities of $\R C$ are {\em Morse singularities},
i.e. the curve can be given (in local coordinates $x,y$ near a singular point) by
equation
\begin{equation}\label{morse}
y^2\pm x^2=0.
\end{equation}

If the $\pm$ sign in \eqref{morse} is $+$ (respectively, $-$),
then the nodal point
is called {\em elliptic} (respectively, {\em hyperbolic}).
The {\em Welschinger sign} of $\R C$
is the product of the signs at all nodal points of $\R C$.
It was shown in \cite{W} that the number
of rational real curves passing
through a generic configuration of $3d-1$ points
and enhanced with these signs
does not depend on the choice of configuration.

It has to be noted that
Welschinger's recipe works only for rational (genus 0) curves. While his 
signs make
perfect sense for real curves in any genus, the corresponding algebraic number
of curves is {\em not} an invariant if $g>0$
(see~\cite{IKS1}).

In \cite{Mik05} it was shown that the expression \eqref{sum-multi} for $\mu_\R$ adds up
to the number of real curves of genus $g$
which are defined by
polynomials with the Newton
polygon $\Delta^*$,
pass through {\em some} generic configuration
of $\varkappa(\Delta) + g - 1$ points in $(\R^\times)^2$,
and are counted with Welschinger's signs.
In the case when $g=0$ and $\Delta^*$ corresponds
to a Del Pezzo surface (e.g. $\Delta^*$ is a triangle with vertices $(0,0)$, $(d,0)$ and $(0,d)$,
corresponding to the projective plane) this result is independent of the choice of
configuration in
$(\R^\times)^2$
by \cite{W}.

It was found in \cite{IKS} that the expression \eqref{sum-multi} is invariant of the choice
of generic {\em tropical} configuration $\X$
for all $g$ and $\Delta$,
even in the cases when the corresponding Welschinger number of real curves is known to be
not invariant. This gives us well-defined tropical Welschinger numbers in situations when the classical
Welschinger numbers are not defined, see \cite{Mik-mfo} for an explanation of this phenomenon.

The multiplicities
proposed by Block and G\"ottsche 
take values in (Laurent) polynomials in one formal variable with positive integer coefficients.
Both $\mu_\C$ and $\mu_\R$ are incorporated in these polynomials
and can be obtained
as their values at certain points.
In the same Block-G\"ottsche multiplicities contain further information. 

In this paper we show that the sum \eqref{sum-multi} of tropical curves enhanced
with the Block-G\"ottsche multiplicities (defined in the next section) is independent of the choice of tropical configuration $\X$. 
In particular, coefficients of this sum at different powers of the formal variable
produce an infinite series of integer-valued invariants of tropical curves complementing
the tropical Gromov-Witten number and the tropical Welschinger number.

We  thank Florian Block, Erwan Brugall\'e and Lothar G\"ottsche for stimulating discussions.

\section{Multiplicities associated to
simple tropical curves in the plane}
\subsection{Definitions}
Let $h: C \to \R^2$ be a properly immersed tropical curve and $V\in C$
is a vertex.
Recall that we denote the dual vector space $Hom_{\R}(\R^2,\R)$
of $\R^2$ with $(\R^2)^*$.

\begin{defn}\label{dual-polygon}
A lattice polygon $$\Delta(V)\subset(\R^2)^*\approx\R^2$$ is called {\em dual} to $V$
if
\begin{itemize}
\item its sides $\Delta_j\subset\dd\Delta(V)$ are parallel to the annihilators
of the vectors $u_h(E_j)$ viewed as linear maps $(\R^2)^*\to\R$ and
\item the integer length $\#(\Delta_j\cap\Z^2)-1$ of $\Delta_j$ 
coincides with the GCD of the coordinates of $u_h(E_j)\in\Z^2$.
\end{itemize}
Here $E_j$ are edges adjacent to $V$ and oriented away from $V$ (the balancing condition
in Definition~\ref{def1.1} guaranties the existence
of such a
polygon)
and $j$ is the index that runs from 1 to the valence of $V$.
\end{defn}

If the immersed tropical curve
$h: C \to \R^2$ is of degree $\Delta$,
then the dual polygons $\Delta(V)$ for all vertices
of $C$ can be placed together in $\Delta^*$
in such a way that they become parts of a certain subdivision $S_h$
of $\Delta^*$.
Each polygon of $S_h$ corresponds
either to a vertex of $C$, or to an intersection point
of images of edges of $C$.
The vertices of $S_h$ are in a one-to-one correspondence
with connected components of $\R^2 \setminus h(C)$.
The subdivision $S_h$
is called {\em dual subdivision} of $h$.

Suppose now that $h$ is simple.
Then every vertex $V$ is 3-valent and thus  $\Delta(V)$ is a triangle.
In this case, the dual subdivision $S_h$ consists
of triangles and parallelograms.

The dual triangle $\Delta(V)$ gives rise to two quantities:
the lattice area $m_\C(V)$ of $\Delta(V)$
and the number $int(V)$ of interior integer points of $\Delta(V)$.
Put $m_\R(V)$ to be equal to $0$ if $m_\C(V)$ is even,
and equal to $(-1)^{int(V)}$ otherwise.
As suggested by Block and G\"ottsche~\cite{BG}, we consider the expression 
\begin{equation}\label{GV}
G_V(y)=\frac{y^{m_\C(V)/2}
- y^{-m_\C(V)/2}}{y^{1/2} - y^{-1/2}}=y^{\frac{m_\C(V)-1}2}+
\dots+y^{\frac{1-m_\C(V)}2}.
\end{equation}
Note that $G_V(1)=m_\C(V)$ and
$G_V(-1)$
is equal to $0$ if $m_\C(V)$ is even, and equal
to $(-1)^{(m_\C(V) - 1)/2}$ if $m_\C(V)$ is odd.

\begin{defn}[\cite{Mik05}]
The numbers
$$
\mu_\C(h) = \prod_V m_\C(V), \;\;\;
\mu_\R(h) = \prod_V m_\R(V),
$$
where each product is taken over all trivalent
vertices of $C$,
are called {\em complex} and {\em real} multiplicities of the simple tropical curve $h$.
\end{defn}

Following~\cite{BG}, we consider a new multiplicity for   $h:C\to\R^2$
\begin{equation}\label{go}
G_h = \prod_V G_V,
\end{equation}
where, once again, the product is
taken over all trivalent
vertices of $C$.
We summarize basic simple properties of $G_h$ in the following proposition.
\begin{prop}
\noindent \begin{enumerate}
\item The Laurent polynomial $G_y$ with half-integer powers is symmetric: $G_h(y)=G_h(y^{-1})$.
\item All coefficients of $G_y$ are positive.
\item We have $G_h(1)=\mu_\C(h)$.
\item If the number of infinite edges $E\subset C$ with even weight $w_h(E)$ is even then
$G_h$ is a genuine polynomial, i.e. all powers of $y$ are integer.
Otherwise all powers of $y$ in $G_h$ are non-integer.
\item If all infinite edge of $E$ have odd weights and the number of infinite edges
$E\subset C$ with $w_h(E)\equiv 3\pmod 4$ is even then $G_h(-1)=\mu_\R(h)$.
\end{enumerate}
\end{prop}
\begin{proof}
These properties hold since $G_h$ is a product of polynomials of the form \eqref{GV}.
The last property is an easy consequence
of Pick's formula, cf.~\cite{IKS}.
\end{proof}

\begin{cor}\label{simple ends}
If $h:C\to\R^2$ is a simple tropical curve such that all of its infinite edges have weight $1$, then
$G_h(y)$ is a symmetric Laurent polynomial with positive coefficients such that
$G_h(1)=\mu_\C(h)$ and $G_h(-1)=\mu_\R(h)$.
\end{cor}

\subsection{Tropical invariance}
Once we have defined multiplicities of simple planar tropical curves we may
consider the number of all tropical curves
of genus $g$ and degree $\Delta$
through a generic configuration $\X$ of
$k = \varkappa(\Delta) + g - 1$ points in $\R^2$
counting each curve with the corresponding multiplicity as in \eqref{sum-multi}.
If the result does not depend on the choice of $\X$ we say that this sum is a {\em tropical invariant}.

As we have already mentioned in the introduction, two multiplicities $\mu_\C$ and $\mu_\R$ introduced in \cite{Mik05}
were known to produce tropical invariants.
The main theorem of this paper establishes such invariance for the Block-G\"ottsche multiplicities $G_h$. 
\begin{theorem}\label{main}
Let
$\Delta \subset \R^2$ be a balanced collection,
$g$ be a non-negative integer number
such that $g \leq \#(\Delta^\circ \cap \ \Z^2)$,
and
$\X\subset \R^2$ be a generic configuration of $k = \varkappa(\Delta) + g - 1$ points.
The sum
$$G(g,\Delta)(y)=\sum\limits_{h\in\S(g,\Delta,\X)}G_h(y)$$
is a symmetric Laurent polynomial in $y$ with positive integer coefficients.
This polynomial is independent on the choice of $\X$.

If $\Delta$
is primitive, we have $G(g,\Delta)(1)$ $=N^\C(g,\Delta)$,
where $N^\C(g,\Delta)$ is the number of complex curves of genus $g$
and of 
Newton polygon $\Delta^*$
which pass through
a generic configuration of $k$ points in $(\C^\times)^2$.
Furthermore, if $\Delta$ is primitive,
there exists a generic configuration $\X^\R$
of $k$ points in $(\R^\times)^2$
such that $G(g, \Delta)(-1) = N^\R(g, \Delta, \X^\R)$, where $N^\R(g, \Delta, \X^\R)$
is the number of real
curves of genus $g$ and of Newton polygon $\Delta^*$
which pass through the points of
$\X^\R$ and are counted with Welschinger's signs.
\end{theorem}

Theorem~\ref{main} is proved in Section~\ref{proof}.

\begin{rem}
If $\Delta$ is non-primitive, then we may interpret $N^\C(g,\Delta)$ as the number of curves in the
polarized toric surface $T_{\Delta^*}$ defined by the polygon $\Delta^*$ that pass through a generic configuration
of $k$ points in $(\C^\times)^2\subset T_{\Delta^*}$ and are subject to a certain tangency condition.
Namely, recall that the sides $\Delta'$ of the polygon $\Delta^*$ correspond to
the divisors $T_{\Delta'}\subset T_{\Delta^*}$.
We require that for each side $\Delta'$ the number of intersection points of the curves we count with
$T_{\Delta'}$ is equal to the number of vectors in the collection $\Delta$ which are dual to $\Delta'$. Furthermore,
all these intersection points should be smooth points of the curves and
we require that for each vector dual to $\Delta'$
the GCD of the coordinates
of the vector coincides with the order of
intersection
of the
curve with $T_{\Delta'}$ in the corresponding point.
We say that such algebraic curves have degree $\Delta$.
\end{rem}

Note that while $N^\C(g,\Delta)$ does not depend on the choice of a generic configuration
of $k$ points in $\tordva$, we do have such dependence
for $N^\R(g, \Delta, \X^\R)$ for $g>0$.
We can strengthen the last statement in the theorem by describing configurations
$\X^\R$
that may be used for computation of $G(g,\Delta)(-1)$.
We refer to \cite{Mik12} for more details. Below we summarize some basic facts about the tropical
number $N^\R(g,\Delta,\X^\R)$ of real curves just for a reference, we will not need these properties
elsewhere in the paper.

\begin{defn}[cf. \cite{Mik12}]
Consider the space $\MM=(\rtordva)^k$ of all possible configurations of (ordered) $k$-tuples
of real points in $\rtordva$. The $(g,\Delta)$-discriminant $$\Sigma_{g,\Delta}\subset\MM$$
(where $g$ is a non-negative integer and $\Delta \subset \R^2$
is a primitive balanced collection)
is the closure of the locus
consisting of configurations $\X^\R\in\MM$ such that there exists a real algebraic curve $\R C$ of degee $\Delta$
passing through $\X$ satisfying to one of the following properties
\begin{itemize}
\item the genus of $\R C$ is strictly less than $g$ (if $\R C$ is reducible over $\C$ then by its genus we
mean $\frac{2-\chi(C)}{2}$, where $C$ is the (normalized) complexification of $\R C$);
\item the genus of $\R C$ is $g$, but $\R C$ is not nodal;
\item
the divisor $3H+K_C-D$ on the complexification $C$ of the real curve $\R C$ is special, where
$H$ is the plane section divisor, $K_C$ is the canonical divisor of $C$, and $D$ is the divisor formed on $\R C$
by our configuration $\X^\R$.
\end{itemize}
\end{defn}

\begin{lem}[cf. \cite{Mik12}]\label{lem12}
\noindent
\begin{enumerate}
\item The $(g,\Delta)$-discriminant is a proper subvariety {\rm (}of codimension at least 1{\rm )} in $\MM$.
\item
If $\X^\R, \Y^\R\in\MM$ are two generic configurations of points
such that
$\X^\R$ and $\Y^\R$ belong to
the same connected component
of $\MM\setminus \Sigma_{g,\Delta}$,
then $N^\R(g,\Delta,\X^\R)=N^\R(g,\Delta,\Y^\R)$.
\item\label{subtropical}
Suppose that $\{p_1, \dots, p_k\}\subset\R^2$
is a {\rm (}tropically{\rm )} generic configuration of $k$ points.
Then for any sufficiently large numbers $t_1,t_2>1$ and any choice of signs $\sigma_j = (\sigma^{(1)}_j, \sigma^{(2)}_j)
= (\pm1,\pm1)\in (\Z/2)^2$, $j = 1$, $\ldots$, $k$,
the configurations
$(\sigma_1 t_1^{p_1},\dots,\sigma_k t_1^{p_k})$ and
$(\sigma_1 t_2^{p_1},\dots,\sigma_k t_2^{p_k})$
are contained in the same connected component
of $\MM\setminus\Sigma_{g,\Delta}$ {\rm (}in particular,
they are disjoint from $\Sigma_{g,\Delta}${\rm )}. Here
$\sigma_j t_i^{p_j} =
(\sigma^{(1)}_jt^{p^{(1)}_j}_i, \sigma^{(2)}_jt^{p^{(2)}_j}_i)$,
where $p_j = (p^{(1)}_j, p^{(2)}_j)$,
$j = 1$, $\ldots$, $k$, and $i = 1, 2$.
\end{enumerate}
\end{lem}

Configurations of points in $(\rtordva)^k$ that can be presented in the form
of \ref{lem12}\eqref{subtropical}  are called {\em subtropical}.

\begin{add}
We have $G(g,\Delta)(-1)=N^\R(g,\Delta,\X^\R)$ for any subtropical configuration $\X^\R\subset\MM$.
\end{add}

The addendum follows from Corollary \ref{simple ends} and Theorem 6 of \cite{Mik05}.

\subsection{Examples}\label{examples}
The polynomials $G(g,\Delta)$ can be computed with the help of floor diagrams
for planar tropical curves \cite{BM}
(particularly with the help of the labeled floor diagrams of \cite{FM})
or with the help of the lattice path algorithm~\cite{Mi-CRAS}.
Each edge of weight $w$ on a floor diagram
contributes a factor of $$(y^{\frac{w-1}2}+y^{\frac{w-3}2}+\dots+y^{\frac{1-w}2})^2$$
to the multiplicity of the floor diagram
as both endpoints of this edge are vertices of multiplicities $w$.

\begin{exa}\label{calculation}
Denote with $\Delta_d$ the primitive
balanced collection of vectors in $\R^2$
such that $(\Delta_d)^*$
is the lattice triangle with vertices $(0,0)$, $(d,0)$ and $(0,d)$.
Note that the projective closure of a curve
in $\tordva$ with Newton polygon $(\Delta_d)^*$ is
a curve of degree $d$ in $\cp^2$. Vice versa, any degree $d$ projective curve disjoint from the points $(1:0:0)$, $(0:1:0)$ and $(0:0:1)$ is uniquely presented as such closure.

We have
$$G(0,\Delta_1)=G(0,\Delta_2)=G(\frac{(d-1)(d-2)}2,\Delta_d)=1.$$
Then we have $G(g,\Delta_d)=0$ whenever $g>\frac{(d-1)(d-2)}2$.

Some other instances of $G(g,\Delta_d)$ are given below:
$$G(0,\Delta_3)=y+10+y^{-1};$$
$$G(2,\Delta_4)=3y+21+3y^{-1};$$
$$G(1,\Delta_4)=3y^2+33y+153+33y^{-1}+3y^{-2};$$
$$G(0,\Delta_4)=y^3+13y^2+94y+404+94y^{-1}+13y^{-2}+y^{-3}.$$
One can easily obtain these formulas from Appendix A (the table) of \cite{FM} listing the floor
diagrams for relevant $g$ and $d$.

E.g. to compute $G(1,\Delta_4)$ we need to look at all
the 13 marked floor diagrams listed in Appendix A. It has 7 labeled diagrams without multiple edges,
the number of corresponding marked floor diagrams (the sum of the $\nu$-multiplicities from the
last column of the table) is 92. Then we have 4 labeled diagrams with a single weight 2 edge
yielding 23 marked floor diagrams; one labeled floor diagram with two weight 2 edges yielding
2 marked floor diagrams and a single marked floor diagram with a weight 3 edge.
We get
\begin{equation*}
G(0,\Delta_4)=92+23(y^{\frac12}+y^{-\frac12})^2+2(y^{\frac12}
+y^{-\frac12})^4+(y+1+y^{-1})^2
\end{equation*}
\end{exa}

Independence of $G(g,\Delta)$ of the choice of
a generic
configuration $\X\subset\R^2$ used for its
computation has implication on the possible multiplicities of tropical curves of genus $g$ and degree $\Delta$
passing through $\X$. For instance, it is well-known (see \cite{Mik05})
that for $g=0$ and $\Delta=\Delta_3$
there are two possible types of a generic configuration of 8 points in $\R^2$.
For one type we have
one tropical curve of complex multiplicity 4 (with two multiplicity 2 vertices connected by an edge,
so its Block-G\"ottsche multiplicity is $y+2+y^{-1}$)
and eight curves of complex multiplicity 1 (so the Block-G\"ottsche multiplicity is also 1).
For the other type we have one curve of complex multiplicity 3 (and the Block-G\"ottsche multiplicity $y+1+y^{-1}$)
and nine curves of complex multiplicity 1. In both cases the total invariant adds up to $G(0,\Delta_3)=y+10+y^{-1}$
and no other distribution of multiplicities is possible.

\subsection{$\delta$-curves}\label{applications}
By the {\em degree} $\deg$ of a symmetric Laurent polynomial
we mean the highest degree of its monomial,
so that e.g.
$$\deg G(1,\Delta_4)=2.$$

For each simple tropical curve $h: \Gamma \to \R^2$,
denote by $\alpha_h$ the degree of the polynomial $G_h$. We refer to $\alpha_h$
as the {\em $\alpha$-multiplicity} of the curve $h$.
Recall that $\varkappa(\Delta^*)-\varkappa(\Delta)$ is the difference between the perimeter
of the integer polygon $\Delta^*$ and the number of vectors in $\Delta$.
If $\Delta$ is primitive then $\varkappa(\Delta^*)-\varkappa(\Delta)=0$.

\begin{prop}\label{degree_curve}
Let
$\Delta \subset \R^2$ be a
balanced collection,
$g$ be a non-negative integer number
such that $g \leq \#(\Delta^\circ \cap \ \Z^2)$,
and
$h: C \to \R^2$ be a simple tropical curve
of genus $g$ and degree $\Delta$.
Then,
$$
\alpha_h \leq \#(\Delta^\circ \cap \ \Z^2) - g + \frac{\varkappa(\Delta^*)-\varkappa(\Delta)}{2},
$$
where $\Delta^\circ$ is the interior of~$\Delta^*$.
Furthermore,
$\alpha_h = \#(\Delta^\circ \cap \ \Z^2)
- g + \frac{\varkappa(\Delta^*)-\varkappa(\Delta)}{2}$
if and only if the dual subdivision
$S_h$ of $\Delta^*$ is formed by triangles.
\end{prop}

\begin{proof}
The statement follows from Pick's formula
applied to the triangles of the dual subdivision $S_h$
of~$\Delta$.
\end{proof}

For any balanced collection $\Delta$ of integer vectors in $\R^2$ and any integer number
$$0 \leq g \leq \#(\Delta^\circ \cap \ \Z^2),$$
we define $$\delta(g, \Delta) = \#(\Delta^\circ \cap \ \Z^2) - g + \frac{\varkappa(\Delta^*)-\varkappa(\Delta)}{2}.$$
If $\Delta$ is primitive, then
$\delta(0, \Delta)$
is the number of interior lattice points
in $\Delta^*$,
and $\delta(g, \Delta)$ is equal to the number
of double points of any nodal irreducible curve in $(\C^\times)^2$
of genus $g$ and of Newton polygon $\Delta^*$.

A simple tropical curve $h: C \to \R^2$
of genus~$g$ and degree~$\Delta$ is called
a {\it $\delta$-curve} (respectively, ($\delta - i$)-{\em curve}) if
$\alpha_h = \delta(g, \Delta)$
(respectively, $\alpha_h = \delta(g, \Delta) - i$).

For a balanced collection $\Delta$ we introduce the number $\pi(\Delta)$
that is equal to the number of ways to introduce a cyclic order on $\Delta$
that agree with the counterclockwise order on the rays in the direction
of the elements of $\Delta$. Clearly, if $\Delta$ is primitive we have $\pi(\Delta) = 1$.
But if $\Delta$ contains non-equal vectors that are positive multiples of each other
then $\pi(\Delta)>1$. E.g. if $\pi(\{(-1,0), (1,3), (0,-1), (0,-2)\})=2$ as there are
two cyclic orders $(0, -1)$, $(0, -2)$, $(1, 3)$, $(-1,0)$ and
$(0, -2)$, $(0, -1)$, $(1, 3)$, $(-1, 0)$ that
agree with the counterclockwise order.

The following proposition was already discovered 
by Block and G\"ottsche  
in the case of primitive $\Delta$ with $h$-transversal $\Delta^*$
(see~\cite{BM2} for the definition of $h$-transversal polygon), in particular for
degrees corresponding to curves in $\pp^2$ and $\pp^1\times\pp^1$.
\begin{prop}[cf. \cite{BG}]
\label{degree}
Let $\Delta \subset \R^2$ be
a  balanced collection,
and $g$ be a non-negative integer number
such that $g \leq \#(\Delta^\circ \cap \ \Z^2)$.
Then,
\begin{enumerate}
\item the degree $\deg G(g, \Delta)$ of $G(g, \Delta)$ is equal to
$\delta(g, \Delta)$;
\item the coefficient of the leading
monomial of $G(g, \Delta)$
is equal to 
$\pi(\Delta)\binom{g+\delta(g, \Delta)}{g}$.
\end{enumerate}
\end{prop}

\begin{proof}
By Proposition~\ref{degree_curve}, one has
$\deg G(g, \Delta) \leq \delta(g, \Delta)$.
A generic configuration
of $k = \varkappa(\Delta) + g - 1$ points
in $\R^2$ can be chosen on a line with irrational slope.
For
such a configuration $\X \subset \R^2$,
the lattice paths algorithm~\cite{Mi-CRAS}
provides a
bijection between certain subsets of
integer points of $\Delta^*$
and the set of $\delta$-curves of genus $g$ and degree $\Delta$
which pass through the points of $\X$.
Here we must restrict to the subsets that contain
all vertices of $\Delta^*$ and exactly
$g$ of the integer points of $\Delta^\circ$.
We have $\binom{g+\delta(g, \Delta)}{g}$
of such choices. The non-vertices points of $\dd\Delta$
must be chosen so that the corresponding curves have degree $\Delta$.
We have $\pi(\Delta)$ of such choices.
These subsets exhaust all paths corresponding to curves of genus $g$ and degree $\Delta$.
Each path produces a unique $\delta$-curve, all other curves for the same path contain at least
one parallelogram in their dual subdivision, so their $\alpha$-multiplicity is strictly smaller than $\delta(g,\Delta)$.
\end{proof}

\begin{cor}\label{bound}
Let $\Delta \subset \R^2$ be a primitive
balanced collection,
and $g$ be a non-negative integer number
such that $g \leq \#(\Delta^\circ \cap \ \Z^2)$.
Then, for any generic configuration $\X\subset\R^2$
of $k = \varkappa(\Delta) + g - 1$ points,
there exist exactly
$\binom{g + \delta(g, \Delta)}{g}$
$\delta$-curves of genus $g$ and degree $\Delta$
which pass through the points of $\X$;
the complex multiplicity of each of these tropical curves
is at least
$1 + 2\delta(g, \Delta)$.

Furthermore, each $\delta$-curve is in a natural 1-1 correspondence with the choice of $g$ points
in $\Delta^\circ\cap\Z^2$.
\end{cor}

\begin{proof}
To establish the lower bound for complex multiplicity, notice
that for a simple tropical curve $h: C \to \R^2$
of complex multiplicity $m$ one has
$\alpha_h \leq \frac{m - 1}{2}$
(the equality being achieved
only if $h$ has a single vertex
of multiplicity greater than $1$).
\end{proof}


\begin{cor}\label{real_count}
Let $\Delta \subset \R^2$
be a
primitive balanced collection,
and $g$ be a non-negative integer number
such that $g \leq \#(\Delta^\circ \cap \ \Z^2)$.
Then, for any sufficiently large positive integer $d$,
the number of real curves
of genus $g$ and of Newton polygon $(d\Delta)^*$
which pass through $k = \varkappa(d\Delta) + g - 1$
points in a subtropical configuration in $(\R^*)^2$
is smaller than $N^\C_{g, d\Delta}$.
{\rm (}Here, $d\Delta$
is the primitive balanced collection obtained
by repeating $d$ times the collection $\Delta$.{\rm )}
\end{cor}

\begin{proof}
The number of interior integer points of~$d\Delta^*$
depends quadratically on $d$. On the other side,
the
number $\varkappa(d\Delta)$
depends linearly on $d$.
Thus, for any sufficiently large integer $d$,
any generic configuration of
$k = \varkappa(d\Delta) + g - 1$ points in $\R^2$ and any  $\delta$-curve $h: C \to \R^2$
of genus $g$ and degree $d\Delta$ which passes through
the points of this configuration,
the number of vertices of $C$ is smaller than
$\frac{2}{3}\delta(g, d\Delta)$. Hence, $C$ has at least
one vertex of complex multiplicity $> 4$ as a curve with $n$ vertices of complex multiplicity at most 4
has $\alpha$-multiplicity at most $n\frac{4-1}{2}$.
The statement now follows from Corollary~\ref{bound}
and Theorem 3 of \cite{Mik05}.
\end{proof}

\begin{prop}
Let $d\ge 7$ be an integer number. For any subtropical configuration $\X$ of $3d-1$ points in $(\R^*)^2$
there exists a rational curve $C$ of degree $d$ in $\cp^2\supset (\C^*)^2\supset (\R^*)^2$ that is not real,
i.e. $\conj(C)\neq C$.
\end{prop}
\begin{proof}
We need to show that for any configuration of $3d-1$ points in tropically general position there exist
a tropical rational curve of degree $d$ (i.e. corresponding to the balanced collection
of vectors $(-1,0)$, $(0,-1)$, $(1,1)$ repeated $d$ times), passing through this configuration
with a vertex of multiplicity different from 1, 2 or 4. Furthermore, if the multiplicity
is 4, then all three adjacent edges must have even weight. Otherwise at least one lift
of this tropical curve is not real by Theorem 3 of \cite{Mik05}.

Suppose that the (unique) $\delta$-curve $C$ conforms to this property.
Since $C$ is rational it has $3d-2$ vertices by Euler's formula.
A vertex adjacent to an infinite ray may not have multiplicity 4 as the weight of the infinite ray is 1.

Note that
by the balancing condition modulo 2
if a vertex is adjacent to an edge of odd weight
then there must be another adjacent edge of odd weight. Thus
if a vertex of $C$ is adjacent to
two infinite rays, then the multiplicity of this vertex is 1.

A vertex of multiplicity 4 contributes $\frac32$ to $\alpha$-multiplicity, a vertex of multiplicity 2 contributes $\frac12$,
while a vertex of multiplicity 1 contributes 0.
As we have $3d$ leaves and each decreases the possible contribution either by 1 or by $\frac34$,
the total $\alpha$-multiplicity of $C$ is bounded from above by
$$\frac32(3d-2)-\frac343d=\frac{9d-12}{4}\le \frac{(d-1)(d-2)}2.$$
The last inequality holds if $d\ge 7$.
\end{proof}

\subsection{Rational $(\delta - 1)$-curves: seven curves in the plane $\pp^2$ and eight in the hyperboloid $\pp^1\times\pp^1$}\label{delta-1}

Proposition~\ref{degree} implies that
for any balanced collection $\Delta \subset \R^2$,
any integer number $0 \leq g \leq \#(\Delta^\circ \cap \ \Z^2)$,
and any generic configuration $\X\subset\R^2$
of $k = \varkappa(\Delta) + g - 1$ points,
there exists a $\delta$-curve of genus $g$ and degree $\Delta$
which passes through the points of~$\X$.
It can happen that all immersed tropical curves of genus $g$
and degree $\Delta$ which pass through
a generic configuration
of $k = \varkappa(\Delta) + g - 1$ points in $\R^2$
are $\delta$-curves. This is the case, for example,
if $g = 0$ and the balanced collection $\Delta$ consists of three vectors,
e.g. $\{(2, -1), (-1, 2), (-1, -1)\}$.

Nevertheless, there are situations, where one can guarantee
the existence of $(\delta-1)$-curves among the interpolating
tropical curves. In particular, we always have rational $(\delta-1)$-curves
in $\pp^2$ and $\pp^1\times\pp^1$ anytime $\Delta$ is primitive and
$\Delta^*$
has lattice points in its interior. Recall that a generic curve of degree $d$
in $\pp^2$ is given by the primitive balanced collection $\Delta_d$
such that $\Delta^*_d$ is the triangle with vertices $(0,0), (d,0), (0,d)$.
Similarly, a generic curve of bidegree $(d,r)$ in $\pp^1\times\pp^1$ is given
by the primitive balanced collection $\Delta_{d,r}$ such that
$\Delta^*_{d,r}$ is the rectangle with vertices $(0,0), (d,0), (0,r), (d,r)$.

\begin{prop}\label{delta-1-curves}
For any generic configuration $\X\subset\R^2$
of $k = 3d - 1$ points, $d \geq 3$,
there exist at least
$7$ rational {\rm (}$\delta - 1${\rm )}-curves of degree $\Delta_d$
which pass through the points of $\X$.
\end{prop}

\begin{proof}
Let $\X\subset\R^2$
be a generic configuration of $k = 3d - 1$ points,
and let $h: C \to \R^2$ be the unique rational
$\delta$-curve of degree $\Delta_d$ such that
$\X \subset h(C)$.
The number of rational $(\delta - 1)$-curves
of degree $\Delta_d$ which pass
through the points of $\X$
is equal to the coefficient of $G(0, \Delta_d) - G_h$ at
$y^{\delta(0, \Delta_d) - 1}$. We need to show that
this coefficient is at least 7, as a $(\delta-1)$-curve can only contribute 1 to
the coefficient $a_{\delta-1}$ of $G(0, \Delta_d)$ at $y^{\delta(0, \Delta_d) - 1}$
and the contribution of higher $\alpha$-multiplicities curves is offset by
the corresponding coefficient of $G_h$.

We can easily compute that  $a_{\delta-1}=3d+1$ with the help of floor diagrams.
To see this we note that a marked floor diagram corresponds to a $(\delta-1)$-curve
if there is an elevator of weight 1 that crosses one floor without stop while all other
elevators connect adjacent floors (or connect the lowest floor to negative infinity).

As the floor diagram is a tree for $g=0$ there are only
two such possibilities: the two top floors are both connected to the third floor from above
or the second floor from below has an infinite elevator going down. The first case
has three possible marking while the second case has $d+2$ markings.

The coefficient of $G_h$ at
$y^{\delta(0, \Delta) - 1} = y^{\frac{(d - 1)(d - 2)}{2} - 1}$
is equal to the number of vertices $V$ of $C$
such that $m_\C(V) > 1$.
The floor-decomposed $\delta$-curve cannot have elevators crossing floors,
so it has $d-2$ elevators of weight 2 or more (those that connect any pair of adjacent floors
except for the one connecting the top two floors). Thus the corresponding $y$-polynomial
multiplicity is a product of $2(d-2)$ non-unit factors and contributes $2(d-2)$ to $a_{\delta-1}$.
Adding up we get $a_{\delta-1}=3+d+2+2d-4=3d+1$.


We can estimate the $y^{\delta(0, \Delta_d) - 1}$-coefficient of $G_h$ for the unique
$\delta$-curve $h:C\to\R^2$ passing through an arbitrary generic configuration $\X$.
The total number of vertices of $C$ is equal to
$3d - 2$. It remains to show that
at least $4$ vertices of $C$ have complex
multiplicity $1$. Denote by $T$ the compliment in $C$
of all infinite edges, and denote by $O$
the set of vertices of $T$ of valency $1$ (in $T$).
If the set $O$ has at least $4$ vertices,
then the required statement is proved, because
any vertex adjacent to two infinite edges of $C$
has complex multiplicity $1$.
If the set $O$ consists of $3$ vertices,
then at least one of these vertices is connected by an edge
to a vertex of $T$ of valence $2$ (in $T$);
the latter vertex is also of complex multiplicity $1$,
and we obtain again that $C$ has at least $4$ vertices
of complex multiplicity $1$.
Finally, assume that the set $O$ consists of two vertices
(the graph $T$ is a tree, thus it has at least two vertices
of valency $1$). In this case, $T$ is a linear tree.
Each of two vertices of valency $1$ in $T$ is connected
by an edge to a vertex of valency $2$ in $T$, and this valency $2$
vertex is of complex multiplicity $1$. Thus, in this case,
the curve $C$ has at least $4$ vertices of complex multiplicity $1$.

Summarizing we see that the $y^{\delta(0, \Delta_d) - 1}$-coefficient of $G_h$
cannot be higher than $3d-2-4$ which is less than $a_{\delta-1}$ by 7.
\end{proof}

\begin{rem}
Notice that the lower bound provided
by Proposition~\ref{delta-1-curves} is sharp in degree $d = 4$.
Indeed, one has
$$G(0,\Delta_4)=y^3+13y^2+94y+404+94y^{-1}+13y^{-2}+y^{-3},$$
see Example~\ref{calculation}. If the generic configuration of
$11 = 3 \times 4 - 1$ points in $\R^2$ is chosen in such a way
that the dual subdivision $S_h$ of the unique
rational $\delta$-curve $h: C \to \R^2$ of degree $\Delta_4$
is the one shown in Figure~\ref{sharpness},
then there exist exactly $13 - 6$
rational $(\delta - 1)$-curves of degree $\Delta_4$
which pass through the points of $\X$, because
the coefficient of $G_h$ at $y^2$ is equal to $6$.
To construct $\X$ it suffices to choose 11 points at distinct edges of a tropical curve dual to
the subdivision of Figure~\ref{sharpness}.
\end{rem}
\begin{figure}[h]
\includegraphics[height=25mm]{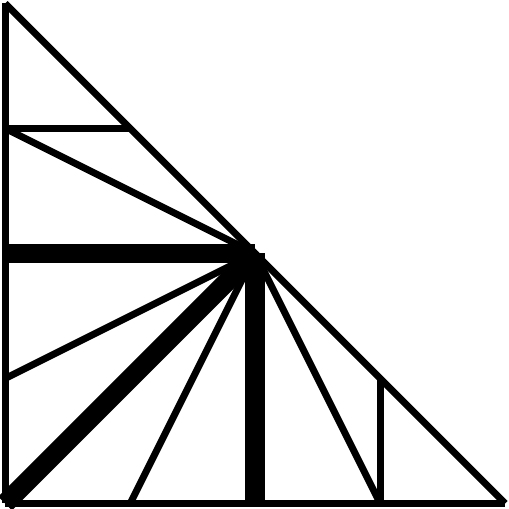}
\caption{\label{sharpness}}
\end{figure}

\begin{prop}\label{delta-1-curves-p1p1}
For any generic configuration $\X\subset\R^2$
of $k = 2d+2r - 1$ points, $d,r \geq 2$,
there exist at least
$8$ rational {\rm (}$\delta - 1${\rm )}-curves of degree $\Delta_{d,r}$
which pass through the points of $\X$.
\end{prop}
\begin{proof}
The proof is similar to the previous proposition. We have $\varkappa(\Delta_{d,r})=2d+2r$
and the infinite directions are either horizontal or vertical, so the maximal contribution of
the only $\delta$-curve for any generic configuration $\X$ to $a_{\delta-1}$ is $2d+2r-6$.
In the same time with the help of floor diagrams we can verify that $a_{\delta-1}=2d+2r+2$,
which is a special case of the following proposition.
\end{proof}

\begin{rem}
The sharpness of Proposition \ref{delta-1-curves-p1p1} is  easy to see for $d=r=2$.
As $G(0,\Delta_{2,2}=y+10+y^{-1}$ and the floor diagram $\delta$-curve have two vertices
of multiplicity 2, we have eight $(\delta-1)$-curves for any floor decomposed generic configuration $\X$.
\end{rem}

Recall (see \cite{BM2}) that an $h$-transversal polygon is given by the following collection
of integer numbers: the length $d_+\ge 0$ of the upper side, the length $d_-\ge 0$ of the lower side,
and two sequences of $d>0$ integer numbers: a non-increasing sequence $d_l$
and a non-increasing sequence $d_r$ (subject to some additional conditions on these numbers,
in particular if $d_+=0$ then the last element of $d_r$ is always greater than the last element of $d_l$).
\begin{prop}\label{adelta-1}
Let $\Delta$ be a primitive balanced configuration such that $\Delta^*$ is an $h$-transverse polygon
that has a lattice point in its interior.
We have
$$a_{\delta-1}=\varkappa(\Delta)-2+c_+(\Delta)+c_-(\Delta)+c_l(\Delta)+c_r(\Delta).$$
for the coefficients $a_{\delta-1}$ of $G(0,\Delta)$ at $y^{\delta(0,\Delta)-1}$.

Here $c_\pm=2$ if $d_\pm>0$. We have $c_+=1$ (resp. $c_-=1$) if $d_+=0$ (resp. $d_-=0$) and the difference between
the last elements of $d_r$ and $d_l$ is 1 (resp. the difference between the first elements of $d_l$ and $d_r$ is 1).
Otherwise $c_\pm=0$.
We define $c_l$ (resp. $c_r$) as the number
of pairs of subsequent elements in the non-increasing sequence $d_l$ (resp. non-increasing sequence $d_r$) that
are different by 1. In particular, we have $a_{\delta-1}\ge\varkappa(\Delta)-2$.
\end{prop}
\begin{proof}
We use the floor decomposition from \cite{BM2}. If $d_+>0$ and $d_->0$
then the contribution of the $\delta$-curve to $a_{\delta-1}$ is $2d-2$ as all its finite elevators must have
weight at least 2. There are two possible floor diagrams for a $(\delta-1)$-curve, they have
$d_++2$ and $d_-+2$ possible markings respectively. Adding up we get
$a_{\delta-1}=2d-2+d_++2+d_-+2=\varkappa(\Delta)+2$ as $\varkappa(\Delta)=d_++d_-+2d$.

If $d_+=0$ the top elevator of the $\delta$-curve floor diagram
may or may not have weight 1. Its weight is equal to the difference
of last elements in $d_r$ and $d_l$. If this weight is 1 then we have a unique $(\delta-1)$-curve
floor diagram with an elevator crossing the second floor from above. In the same time the contribution
of the $\delta$-curve to $a_{\delta-1}$ has to be decreased by 2 in such case (in comparison with
$2d-2$ in the case when all $(d-1)$ finite elevators have weight at least 2).
If this weight is 2 then there is no correction neither to the contribution of the $\delta$-curve nor
to the number of $(\delta-1)$-curves. We have a similar situation for the case $d_-=0$.
\end{proof}

\begin{exa}
If $\Delta=\Delta_d$ we have a lattice point in the interior of $\Delta^*$ iff $d\ge 3$.
In this case we have $$a_{\delta-1}=3d+1=\varkappa(\Delta)+1.$$
If $\Delta=\Delta_{d,r}$ we have a lattice point in the interior of $\Delta^*$ iff $d,r\ge 2$.
In this case we have $$a_{\delta-1}=2d+2r+2=\varkappa(\Delta)+2.$$
\end{exa}

If $\Delta^*$ is a general $h$-transverse polygon, the  argument from the proofs
of Proposition \ref{delta-1-curves} and \ref{delta-1-curves-p1p1} that ensures two multiplicity 1 vertices
for a $\delta$-curve is not applicable. But the contribution of the $\delta$-curve to $a_{\delta-1}$
can still be bounded from above by $\varkappa(\Delta)-2$, the number of all vertices of a rational curve
with $\varkappa$ tails. Thus we get the following corollary.
\begin{cor}
Let $\Delta$ be a primitive balanced configuration such that $\Delta^*$ is an $h$-transverse polygon
that has a lattice point in its interior. For any generic configuration $\X$ of $\varkappa(\Delta)-1$ points in $\R^2$
there exist at least $c_+(\Delta)+c_-(\Delta)+c_l(\Delta)+c_r(\Delta)$ distinct $(\delta-1)$ rational curves through $\X$,
where $c_\pm$, $c_r$ and $c_l$ are defined in Proposition \ref{adelta-1}.
\end{cor}

\begin{rem}
For any non-negative integer $j$,
we may treat the coefficient
of the polynomial $G(g,\Delta)$ at $y^j$ as a non-negative integer
invariant $a_j(g,\Delta)$
for the number of tropical curves passing through a generic configuration of $k = \varkappa(\Delta) + g - 1$ points.
Here only tropical curves of multiplicity at least $2j+1$ contribute to $a_j(g, \Delta)$
(each with the corresponding coefficient
at $y^j$ of its Block-G\"ottsche multiplicity).
Thus $G(g,\Delta)$ can be viewed as infinite number of integer-valued invariants
of tropical curves.
\end{rem}

\section{Proof of Theorem \ref{main}}\label{proof}
The second part of the statement follows from
Theorem 6 of \cite{Mik05}.
The fact that $G(g, \Delta)$ is a symmetric Laurent polynomial
with positive coefficients immediately follows
from the definition. It remains to prove
that $G(g, \Delta)$ is independent on the choice of $\X$.
Our proof is similar to the proof of Theorem 4.8 in~\cite{GM}
and the proof of Theorem 1 in~\cite{IKS}.

Recall that $\X=\{p_1,\dots,p_k\}$ is a configuration of  $k = \varkappa(\Delta) + g - 1$ points in $\R^2$
tropically generic in the sense of Definition 4.7 of \cite{Mik05}.
To show independence of $G(g,\Delta)$ of $\X$ it suffices to show that the sum \eqref{go} stays invariant
if we move one of the points of $\X$, say $p_k$ in a smooth path $p_k(t)$, $t\in[-\epsilon,\epsilon]$, $\epsilon>0$,
so that the configurations
$\X(t) = \{p_1,\dots,p_{k-1},p_k(t)\}$ are tropically generic
whenever $t\neq 0$.

Let $t_0 \in [-\epsilon, \epsilon]$,
and let $h(t_0): C(t_0) \to \R^2$ be a tropical curve
of genus $g' \leq g$ and degree $\Delta$
such that $\X(t_0) \subset h(t_0)(C(t_0))$. Put
\begin{equation}\label{deff}
\deff(C(t_0)) = \sum_V (\val(V) - 3) + (g-g') + m,
\end{equation}
where the first sum is taken over all vertices of $C(t_0)$,
and $m$ is equal
to the number of vertices of $C(t_0)$ whose images
under $h(t_0)$ are contained in $\X(t_0)$.

As
$\X(t_0)$ is tropically generic in the case
$t_0 \ne 0$, it follows from Proposition 2.23 of \cite{Mik05}
that $\deff(C(t_0)) = 0$. Furthermore, if we slightly perturb our
generic points $p_1, \dots, p_{k-1}\in\R^2$
the sets $\S(g,\Delta,\X(\pm\epsilon))$ remain unchanged after
perturbation.
Proposition 3.9 of~\cite{GM} implies the following statement.

\begin{lem}\label{codim1}
There exists a finite set $D \subset \R^2$
such that under the condition $p_k(t_0) \not\in D$
one has either $\deff(C(t_0)) \leq 1$, or
$\deff(C(t_0)) = 2$ and $C(t_0)$ has two $4$-valent vertices
connected by two edges.
\end{lem}

\begin{proof}
Proposition 3.9 of \cite{GM}
concerns the dimensions of the moduli spaces
$\MM_{g,\Delta}^\alpha$,
of tropical curves $h:C\to\R^2$ of genus $g'\le g$ and degree $\Delta$ with
$k$ marked points $y_1,\dots,y_k\in C$
such that $(h, y_1, \ldots, y_k)$
has a combinatorial type $\alpha$. Here,
by the combinatorial type
of $(h,y_1,\dots,y_k)$, we mean
the combinatorial type of the graph $C$ together with
the slopes of its edges under $h$ and the distribution of the points $y_1,\dots,y_k$
among the edges and vertices of $C$.

We are looking at the curves $h: C \to \R^2$ such that
$h(y_j)=p_j$, $j = 1, \dots, k$.
For a given combinatorial type $\alpha$,
consider the evaluation map $$\ev\MM_{g,\Delta}^\alpha\to(\R^2)^k$$
defined by $(h,y_1,\dots,y_k)\mapsto (h(y_1),\dots,h(y_k))$.
As we may slightly perturb our generic points $p_j$, $j=1,\dots,k-1$, if needed,
we may assume that $\ev^{-1}((p_1, \dots, p_{k-1})\times\R^2)$ is of codimension
$2k-2$ in $\MM_{g,\Delta}^\alpha$. Thus,  any curve
$h: C\to\R^2$ with $h(y_j) = p_j$, $j = 1$, $\ldots$, $k$,
must be of a combinatorial type $\alpha$ with
$\dim\MM_{g,\Delta}^\alpha\ge 2k-2$. Furthermore, each $\alpha$
with $\dim\MM_{g,\Delta}^\alpha = 2k-2$ has a most one (by convexity)
value $p_k$ admitting such $h$. There are only finitely many distinct
combinatorial types.
Thus away from a finite set $D\subset\R^2$ we only encounter
combinatorial types~$\alpha$ with $\dim\MM_{g,\Delta}^\alpha > 2k-2$
and they are explicitly described by Proposition 3.9 of \cite{GM}.
\end{proof}



By Lemma \ref{codim1} we may assume that the path
$\X(t)$, $t\in[-\epsilon,\epsilon]$, is such that for any curve $h(t_0):C(t_0)\to\R^2$ of degree $\Delta$ and genus $g'\le g$
passing through $\X(t)$ we have
$\deff(C) \leq 2$.
In addition, we have $\deff(C(t_0))=0$ whenever $t_0\neq 0$
and $\deff(C(0))\le 1$ unless $C(0)$ has two 4-valent vertices connected by two edges.

Suppose that $h(\epsilon):C(\epsilon)\to\R^2$ is a curve passing through $\X(\epsilon)$.
When we change $t$ from $\epsilon$ to $0$ the configuration $\X(t)$ moves as well in
the class of generic configurations.
This uniquely defines a continuous deformation $h(t):C(t)\to\R^2$ as by Lemma 4.20 of \cite{Mik05}
every connected component
$T(\epsilon) \subset C(\epsilon) \setminus h^{-1}(\X(\epsilon))$
is a 3-valent tree with a single leaf going to infinity.
All the other leaves of $h(T(\epsilon))$ are adjacent to some points of $\X(\epsilon)$.

Thus one can reconstruct $h(t):C(t)\to\R^2$, $0<t<\epsilon$, by tracing the change of $h(T(\epsilon))$ for each such component $T(\epsilon)$.
We do it inductively. If $T(\epsilon)$ is a tree without 3-valent vertices then $h(T(\epsilon))$ is an open ray adjacent to $p_j(\epsilon)$.
If $j<k$ this point does not move and $T(\epsilon)$ remains constant under the deformation. If $j=k$, then $h(T(\epsilon))$ deforms to a parallel ray
emanating from $p_k(t)$.

Suppose that $T(\epsilon)$ contains 3-valent vertices.
Unless $h(T(\epsilon))$ is adjacent to $p_k(\epsilon)$, it remains constant under deformation
as its endpoints do not move.
Let $E(\epsilon)$ be the edge of $T(\epsilon)$ connecting $p_k(\epsilon)$ to a 3-valent vertex $v\in T(\epsilon)$.
The complement $T(\epsilon)\setminus\{v\}$ consists of three components: the edge $E(\epsilon)$ and two other components
$T_0(\epsilon)$ and $T_\infty(\epsilon)$, where $T_\infty(\epsilon)$ is chosen so that it contains the infinite edge leaf.

Let $E_0(\epsilon)$ be the edge of $T_0(\epsilon)$ adjacent to $v$. The line parallel to $E(\epsilon)$ passing through $p_k(t)$
intersects the line containing $E_0(\epsilon)$ at a point $v(t)$. If $t<\epsilon$ is sufficiently close to $\epsilon$
then $v(t)$ is sufficiently close to $v$. 
We form $h(T(t))$ by taking the union of the interval connecting $p_k(t)$ to $v(t)$ and the tree obtained
by modifying $T_0(\epsilon)$ by enlarging or decreasing its leaf edge adjacent to $v(\epsilon)$ so that
$h(t)(T_0(t))$ is adjacent to $v(t)$. Then we modify the component $T_\infty(\epsilon)$ inductively by treating
the vertex $v(\epsilon)$ as the marked endpoint for this tree, see Figure \ref{deform}.

\begin{figure}[h]
\includegraphics[height=35mm]{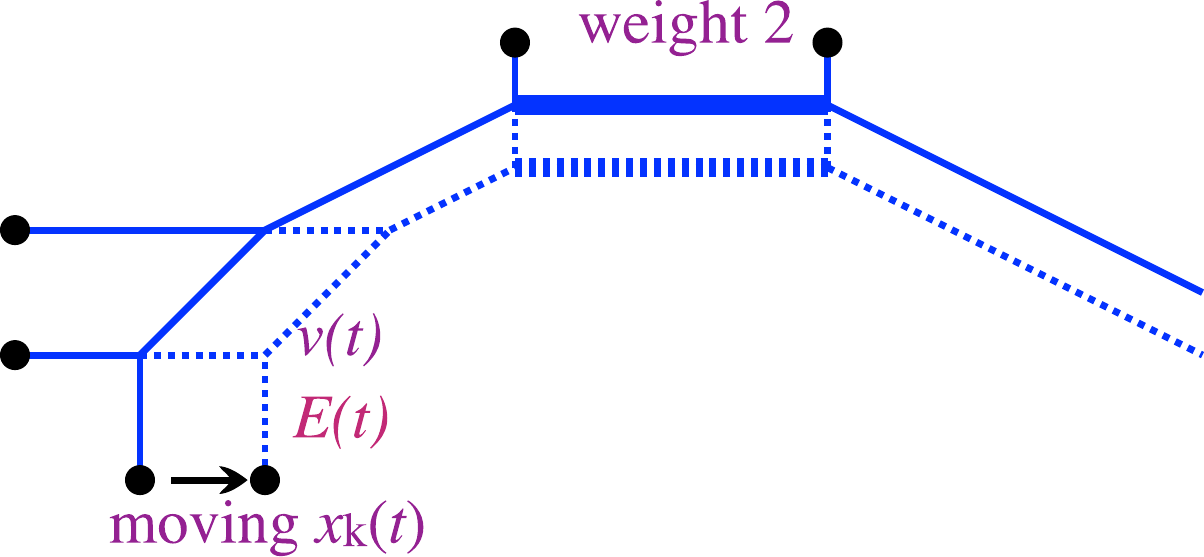}
\caption{\label{deform} Deformation of $T(\epsilon)$}
\end{figure}

Note that we may continue this deformation $h(t):C(t)\to\R^2$ for any value of $t$, $0<t<\epsilon$.
Indeed, the set of $t\in (0,\epsilon]$ for which such a deformation exists is an open neighborhood of $\epsilon$.
Let $t_{\inf}$ be the infimum of this set. When $t\to t_{\inf}^+$ we get the limiting tree $h(T(t_{\inf}^+))$ for
each component $T(\epsilon)\subset C(\epsilon)\setminus h^{-1}(\X(\epsilon))$. This tree is a degeneration
of the combinatorial type of $T(\epsilon)$ as the length of the edges of $T(\epsilon)$ changes and some
values in the limit $t\to t_{\inf}^+$ might become zero.

Note that if a length of an edge of $T(\epsilon)$
vanishes then either two or more trivalent vertices collide to a vertex of higher valence or one of the
3-valent vertex collides with a point of $X(t_{\inf}^+)$. We may combine a limiting curve
$h(t_{\inf}^+):C(t_{\inf}^+)\to\R^2$ by taking the union of the limiting trees for all such component.
Note that the degree of the limiting curve is still $\Delta$ as the number and direction of the infinite rays
do not change.

\begin{lem}\label{limit-genus}
The genus of $C(t_{\inf}^+)$ is $g$.
\end{lem}
\begin{proof}
The genus of limiting curve may only decrease if the length of all edges in a cycle of $C(\epsilon)$
will simultaneously vanish. This is not possible by Lemma \ref{codim1} as our path for $p_k(t)$ is chosen
to avoid the set $D$.
\end{proof}

Note that $\deff(C(t_{\inf}^+))$ coincides with the number of the vanishing edges (for all components
of $C(\epsilon) \setminus h^{-1}(\X(\epsilon))$). Thus $t_{\inf}=0$.
Similarly we may deform any curve $h(-\epsilon):C(-\epsilon)\to\R^2$ from $\S(g,\Delta,\X(-\epsilon))$ to a limiting curve
$h(0^-):C(0^-)\to\R^2$.

Theorem \ref{main} now follows from the following Lemma.
\begin{lem}
For each immersed tropical curve $h: C \to \R^2$
such that $\X(0) \subset h(C)$, we have
\begin{equation}\label{g-sum}
\sum G(h^+_j)=\sum G(h^-_j)
\end{equation}
where  $h^{\pm}_j$ 
runs over all curves $\S(g,\Delta,\X(\pm\epsilon))$ such that
the limiting curve $h^{\pm}_j(0^{\pm})$ coincides with~$h$.
\end{lem}

\begin{proof}
By Lemma \ref{limit-genus} we may assume that the genus of $C$ is $g$ as otherwise
the sums in both sides of \eqref{g-sum} are empty.
By Lemma \ref{codim1} we only need to consider the case when $\deff(C)=1$ and
the exceptional case of $C$ with two 4-valent vertices connected by two edges.

We assume that $h$ can be presented as the limiting
curve of a continuous family $h(t):C(t)\to\R^2$, $0<t<\epsilon$ (by changing the parameter
$t\mapsto-t$ if needed).
First we consider the case when $\deff(C)=1$, $C$ is 3-valent and $m=1$ (see \eqref{deff}).
In this case we have a 3-valent vertex $v\in C$ such that $h(v)=x_j(0)$, for some $j=1,\dots,k$.
Accordingly, the length of the
segment $E(t)$ connecting $(h(t))^{-1}(p_j(t))$ to a 3-valent vertex $v(t)$
in a component $T(t)\subset C(t)\setminus h^{-1}(\X(t))$ must vanish.

Let~$A$ be the connected component  of $(C\setminus h^{-1}(\X(0)))\cup\{v\}$
that contains the point $v$.
Note that~$A$ comes as the union of the limits
of the family of components $T(t)$ and
the family of components $T'(t)$ adjacent to $p_j(t)$ from the other side, $0<t\le\epsilon$
(note that $T(t)$ may coincide with $T'(t)$
as $T(t)\cup\{p_j(t)\}$ does not have to be a tree).

Similarly to the situation we have considered above,
the complement $T(t)\setminus\{v(t)\}$ consists of three connected components:
$E(t)$, $T_0(t)$ and $T_\infty(t)$, where $T_\infty$ is the component containing
the edge going to infinity. Again we denote with $E_0(t)$ the edge of $T_0(t)$ adjacent to $p_j(t)$
(and with $E_0(0^+)$ the limit of this edge when $t\to 0^+$).  We denote with $E_\infty(t)$
the edge of $T_\infty(t)$ adjacent to $p_j(t)$. Note that while the length of all these edges as well
as its position in $\R^2$ depend on $t$, their slope remains constant.

Let $L$ be the line extending $E_0(0^+)$.
The points $p_j(t)$, $t>0$ sit in the same half-plane $H$ bounded by $L$ (since $\X(t)$ is tropically
generic whenever $t\neq 0$).
If the points $p_j(t)$, $t<0$, sit in the same half-plane $H$ then we may extend the family
$h(t):C\to\R^2$, $0<t\le\epsilon$,
to $-\epsilon<t<0$ keeping the same combinatorial type by the same reconstruction procedure.

Suppose that $p_j(t)$ sit in the other half-plane for $t<0$ (note that in such case this
holds for all $-\epsilon\le t<0$).
For $t<0$ we define $T_0(t)=E'(t)\cup T_0(0^+)$, where $E'(t)$ is the interval
connecting  $p_j(t)$ to $L$ and parallel to $E_\infty(\epsilon)$, see Figure \ref{3v}.
The remaining components of $S\setminus\{v\}$ (as well as those of $C\setminus(S\cup h^{-1}(\X(0)))$)
are trees without vanishing edges,
so they deform to negative values of $t$ as before.
\begin{figure}[h]
\includegraphics[width=125mm]{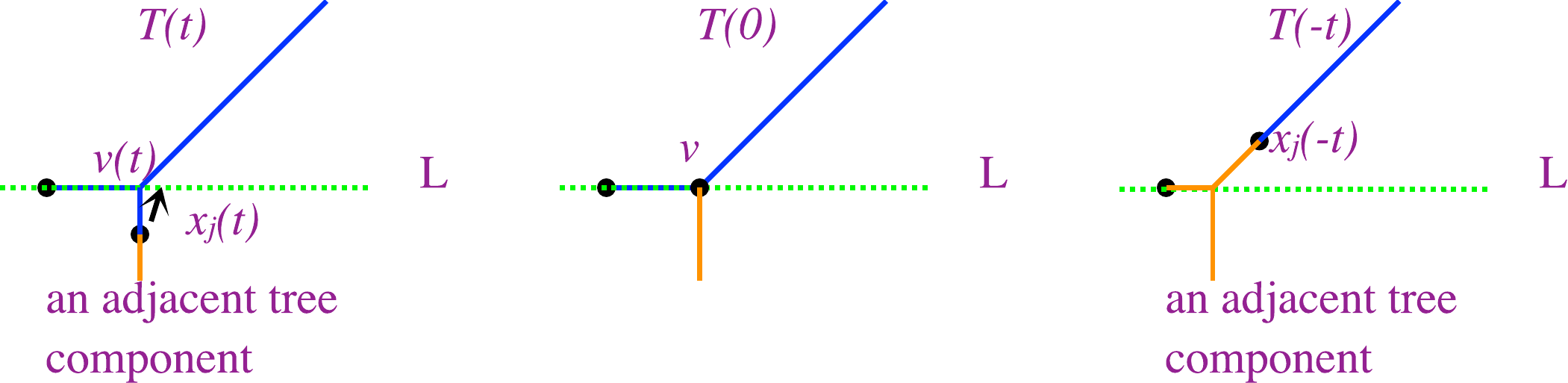}
\caption{\label{3v} Collision of $v(t)$ and $p_j(t)$}
\end{figure}

This shows that $h$ can be presented
as the limiting curve $h(0^-)$ for a family $h(t)\in\S(g,\Delta,\X(t))$ for $t\in [-\epsilon,0)$.
Its combinatorial type is uniquely determined, so by Lemma 4.22 of \cite{Mik05} the family $h(t)$
is unique and both sums in \eqref{g-sum} consist of a unique term.
These terms have the same multiplicity as the curves $h(t):C(t)\to\R^2$ for $\pm t>0$ have the same
multiplicities for their vertices as the slopes of the corresponding edges are the same
(in fact, the only difference of their combinatorial types is in the edge containing $p_j(t)$).

Let us now consider the case when $\deff(C)=1$, $m=0$ (see \eqref{deff}), and a vertex $v\in C$ is 4-valent.
This corresponds to the case when the length of the edge $E(t)\subset C(t)$ connecting
two vertices $v(t), v'(t)\in C(t)$, $t\in (0,\epsilon]$ vanishes. Consider the component~$A$
of $C\setminus h^{-1}(\X(0))$ containing the vertex $v$. This component 
is a tree since $\deff(C)=1$ and thus no edges of $C(t)$, $t\in (0,\epsilon]$, except for $E(t)$ may vanish.

Denote the edges of $C$ adjacent to $v$ with $E_1$, $E_2$, $E_3$ and $E_\infty$,
so that the order agrees with the counterclockwise order around $h(v)\in\R^2$ and $E_\infty$ is chosen from the
component of $A \setminus v$
containing an edge going to infinity, see Figure \ref{4v-vertex}.

\begin{figure}[h]
\includegraphics[width=80mm]{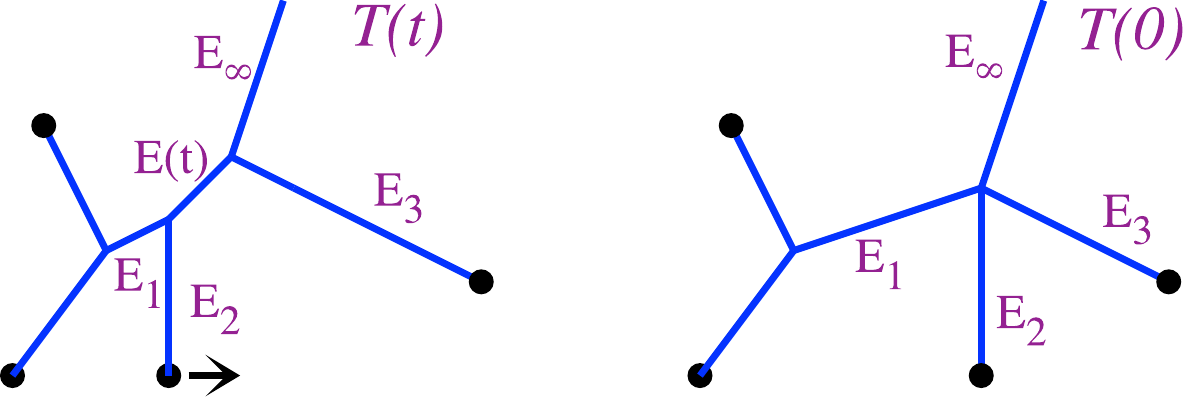}
\caption{\label{4v-vertex} Collision of $v(t)$ and $v'(t)$ to a 4-valent vertex.}
\end{figure}

Each edge $E_j$ must come as the limit of an edge $E_j(t)$ of the approximating curve $C(t)$.
We denote the endpoint of $E_j$, $j=1,2,3$, not tending to $h(v)$ by $v_j(t)\in C(t)$, $t>0$.
(Note that $E_\infty$ might not have the other endpoint as it might happen to be an unbounded edge.)
The point $h(v_j(t))\in\R^2$ is inductively determined by $\X(t)$ as well as the slopes of the edges of $C$.
These points are thus well-defined also for negative values of $t\in [-\epsilon,\epsilon]$.
Denote with $R_j(t)$, $j=1,2,3$, the rays emanating from the points $h(v_j(t))$ in the direction of the
edges $E_j(t)$. For $t\neq 0$ these edges cannot intersect in a triple point as $\X(t)$ is generic,
but their pairwise intersections must remain close enough to a triple intersection.

If there are no parallel rays among $R_j$ we have one of the two types depicted on Figure \ref{triangles}.
If  the configuration $\X(t)$ for $\pm t>0$ corresponds to the same type of intersection
then the combinatorial types of curves from $\S(g,\Delta,\X(\pm\epsilon))$ coincide
and both sums in \eqref{g-sum} are literally the same. Thus we may assume that we have
different types of intersections for different signs of $t$.
\begin{figure}[h]
\includegraphics[width=100mm]{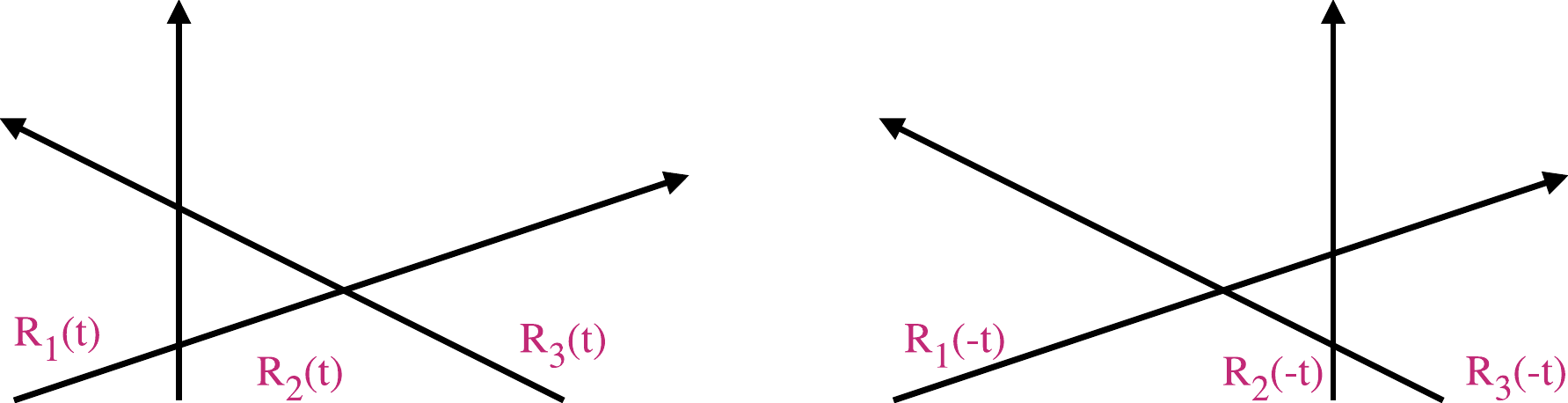}
\caption{\label{triangles} Intersection of rays $R_1(t), R_2(t), R_3(t)$ extending the edges $E_1(t), E_2(t), E_3(t)$.}
\end{figure}

Possible ways to extend $R_j$ to get a 3-valent perturbation of the neighborhood of the 4-valent
point $v\in C$ are depicted on Figure \ref{res-triangles}. We see that we have three possible types
for such perturbation. Without loss of generality we may assume that two of them
correspond to $t>0$ and one corresponds to $t<0$.
\begin{figure}[h]
\includegraphics[width=100mm]{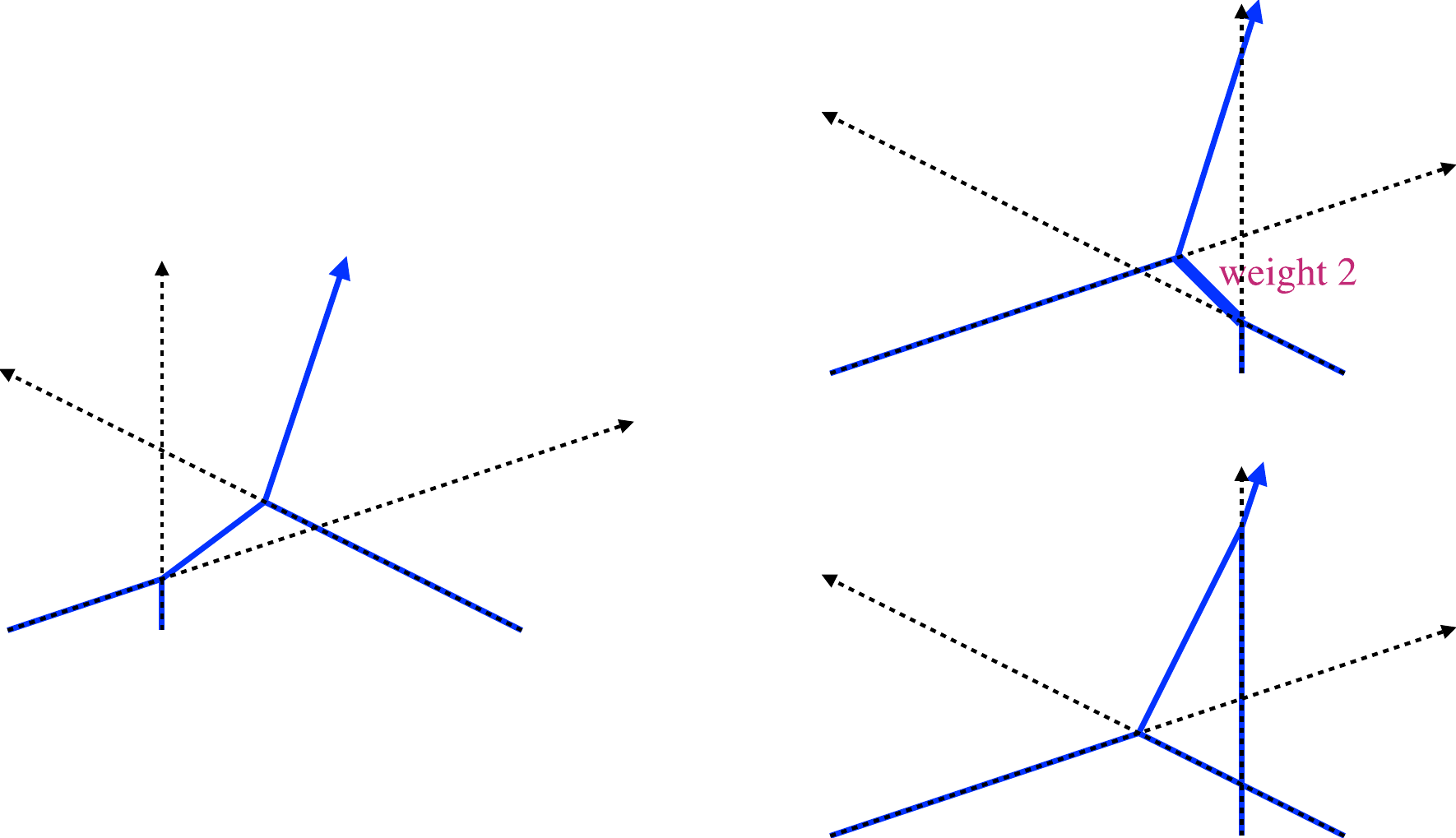}
\caption{\label{res-triangles} Possible ways to perturb a 4-valent vertex}
\end{figure}

To compare the contribution of these perturbations to the corresponding sides of \eqref{g-sum}
we consider the dual quadrilateral $Q$ to the vertex $v$, see Definition \ref{dual-polygon}.
Each type of perturbation of $C$ where a 4-valent vertex $v$ is replaced by two
trivalent vertices defines a subdivision of $Q$ into two triangles and, possibly, a parallelogram
(which corresponds to the case when there is a self-intersection point of $C(\pm\epsilon)$ near
$h(v)$), cf. section 4.1. of \cite{Mik05}.

The subdivisions dual to the three possible types
of perturbation are shown of Figure \ref{3subdivisions}.
Two of these subdivisions are given by drawing diagonals.
If the quadrilateral $Q$ does not have parallel sides (i.e. no rays $R_j$, $j=1,2,3,\infty$ are parallel)
then the third subdivision maybe described as follows.
There is a unique parallelogram $P$ such that two of the sides of $P$ coincide with two
of the sides of $P$ and $P\subset Q$. The complement $Q\setminus P$ splits into two triangles.
Note that the third triangle corresponds to the same sign of $t$ as the subdivision given by the diagonal of $Q$
that also serves as a diagonal of $P$.
\begin{figure}[h]
\includegraphics[width=80mm]{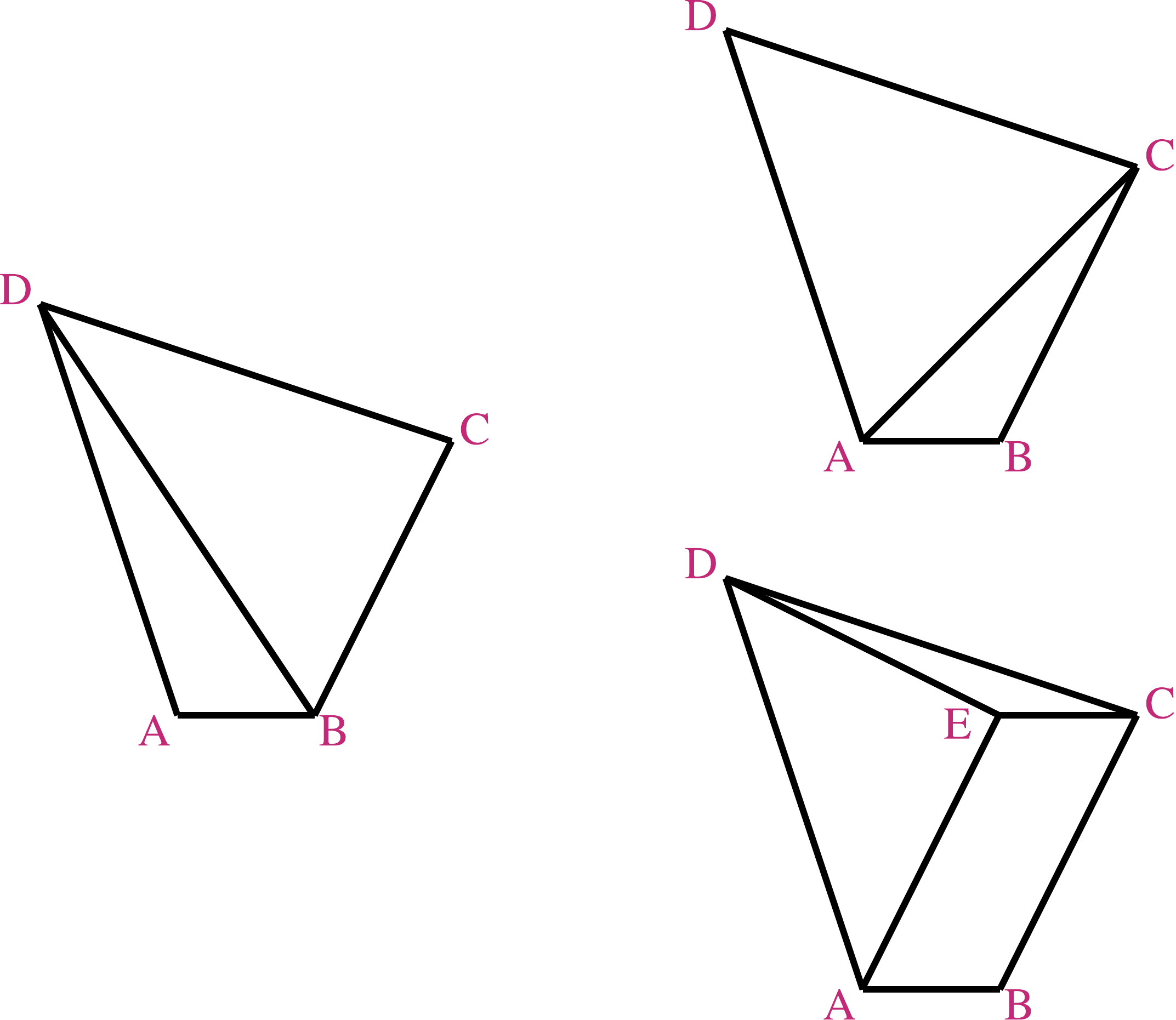}
\caption{\label{3subdivisions} Subdivisions dual to resolving a 4-valent vertex if no adjacent edges are parallel.}
\end{figure}

We are ready to compute both sides of \eqref{g-sum} for the case when $Q$ does not have parallel sides.
We denote the vertices of $Q$ by $A, B, C, D$ in the counterclockwise order
so that the sides $AB$ and $BC$ are also the sides of the parallelogram $P$. Let $E$ be the fourth vertex
of $P$. Without loss of generality we may assume that $E$ is inside the (closed) triangle $BCD$
(if $E$ is on the diagonal $AD$ then treat $BDE$ as a degenerate triangle so that $\Area(BDE)=0$).

We have the following straightforward area inequalities: $\Area ACD>\Area ABC$,
$\Area BCD)\ge\Area ABD$ and $\Area ADE\ge\Area CDE$. In the last two inequalities we have
equalities if and only if the triangle $BDE$ is degenerate.


In the right-hand side of \eqref{g-sum} we have a single term proportional to
\begin{eqnarray*}
(y^{\Area BCD}-y^{-\Area BCD})(y^{\Area ABD}-y^{-\Area ABD})=\\
y^{\Area Q}-y^{\Area BCD-\Area ABD}-y^{-\Area BCD+\Area ABD}+
y^{-\Area Q}.
\end{eqnarray*}
The proportionality coefficient here is the product of Block-G\"ottsche multiplicities of all other
vertices of $C$ divided by $(y^{\frac12}-y^{-\frac12})^2$.
The left-hand side of \eqref{g-sum} is the sum of two terms:
\begin{eqnarray*}
(y^{\Area ACD}-y^{-\Area ACD})(y^{\Area ABC}-y^{-\Area ABC})=\\
y^{\Area Q}-y^{\Area ACD-\Area ABC}-y^{-\Area ACD+\Area ABC}+
y^{-\Area Q}
\end{eqnarray*}
and
\begin{eqnarray*}
(y^{\Area ADE}-y^{-\Area ADE})(y^{\Area CDE}-y^{-\Area CDE})=
y^{\Area ADE+\Area CDE}\\
-y^{\Area ADE-\Area CDE}-y^{-\Area ADE+\Area CDE}+
y^{-\Area ADE - \Area CDE}
\end{eqnarray*}
with the same proportionality coefficient.
The terms $y^{\Area Q}+y^{-\Area Q}$ in both sides annihilate.
Furthermore we have a simplification after adding the two terms
of the left-hand side as $$\Area ACD - \Area ABC=\Area Q-\Area P=\Area ADE + \Area CDE.$$

We are left with $-y^{\Area BCD-\Area ABD}-y^{-\Area BCD+\Area ABD}$ in the right-hand side
and with $-y^{\Area ADE-\Area CDE}-y^{-\Area ADE+\Area CDE}$. Thus to finish the proof in the
case when $Q$ is not a trapezoid it suffices to show that $\Area BCD - \Area ABD = \Area ADE - \Area CDE$.
However, since $P=ABCE$ is a parallelogram we have
$$\Area BCD - \Area ADE = \Area P = \Area ABD - \Area CDE.$$

Let us now consider the case when some of the edges $E_1, E_2, E_3, E_\infty$ are parallel.
Note that in this case only two of them may be parallel to the same direction.
Indeed, by the balancing condition if three edges are parallel then the fourth edge must also
be parallel to them. But in this case $v$ is the result of colliding of two 3-valent vertices
of $C(t)$ where the map $h(t)$ cannot be an immersion. Thus either we have exactly two
edges that are parallel or we have two pairs of parallel edges.

Suppose that two edges are not only parallel, but emanate from $v$ in the same direction.
By the balancing condition two other edges cannot be parallel. If one of the two parallel edges is $E_\infty$
then no two rays among $R_1(t), R_2(t), R_3(t)$ are parallel. So, once again we have one of the two
ways to perturb a triple point of intersections of these rays (cf. Figure \ref{triangles}).
If the pair of parallel edges is disjoint from $E_\infty$ then there are still two possibilities for
the rays $R_1(t) R_2(t), R_3(t)$ as the parallel rays may be perturbed in two different ways.

In any of these cases
the dual polygon $Q$ in this case is a triangle. Furthermore any nearby configuration of $R_1(t), R_2(t), R_3(t)$
corresponds to a unique subdivision of the triangle $Q$ into triangles so that the new vertex of the
subdivision is contained in the side dual to the pair of parallel edges  and subdivides them into
the intervals of integer lengths corresponding to the weights of the parallel edges. The only possible difference
is the order of these intervals in $\dd Q$. The unordered pair formed by the areas of the  triangles
of the subdivision is the same, thus the corresponding Block-G\"ottsche multiplicities are also the same. 

If there are no edges among $E_1, E_2, E_3, E_\infty$ emanating in the same direction,
but there are parallel edges then $Q$ is a trapezoid (possibly a parallelogram as we may have two
pairs of parallel edges in this case).
Then there is a unique way to reconstruct a perturbation
of $h:C\to\R^2$ for each of the two cases of Figure \ref{triangles} as the combinatorial
type of one of the perturbations (the one with a self-intersection point)
has a 3-valent vertex with all three adjacent edges parallel to the same direction.
This combinatorial type cannot be realized by an immersion and thus does not appear
for a generic configuration of points $\X(t)$, $t\neq 0$.

Thus if $Q=ABCD$ is a trapezoid (say $AB$ and $CD$ are parallel sides)
we have the contribution of $y^{\Area Q}-y^{\Area BCD-\Area ABD}-y^{-\Area BCD+\Area ABD}+
y^{-\Area Q}$ and of $y^{\Area Q}-y^{\Area ACD-\Area ABC}-y^{-\Area ACD+\Area ABC}+
y^{-\Area Q}$ on the different side of \eqref{g-sum}. But $\Area BCD=\Area ACD$ while
$\Area ABC=\Area ABD$ since $AB$ and $BC$ are parallel, so the contributions are the same.

Finally we have to consider the case when $\deff(C)=2$. Let $v,v'\in C$ be two 4-valent
vertices connected by two edges $E,E'\subset C$. Note that by Lemma \ref{codim1} the vertices of $C$ are disjoint
from $\X(0)$. We claim that if $h:C\to\R^2$ can be presented as a limiting curve $h(0^\pm)$
for $h(t)\in\S(g,\Delta,\X(t))$, $\pm t>0$, then $h(E\cup E')\cap\X(0)\neq\emptyset$.
Indeed, the union $E\cup E'$ forms a cycle in $C$ and if it is disjoint from $\X(0)$ it must
remain disjoint from $\X(t)$ after a perturbation which contradicts to our hypothesis that $\X(t)$, $t\neq 0$,
is generic by Lemma 4.20 of \cite{Mik05}.

On the other hand the set $h(E\cup E')\cap\X(0)$ cannot have more than two points as each edge
of $C(t)$, $t\neq 0$, can hit no more than one point of $\X(t)$.
If we have two points $p_j(0), p_{j'}(0)\in h(E\cup E')\cap\X(0)$ then they must come from different edges of
the approximating curve, i.e. $p_j(t)\in E(t)$ and $p_{j'}(t)\in E'(t)$, where $E(t),E'(t)\subset C(t)$ are
the edges limiting at $E$ and $E'$.

The (common) endpoints $v(t),v'(t)$ of $E(t)$ and $E'(t')$ belong to two different tree components
$T(t)$ and $T'(t)$ of $C\setminus h(t)^{-1}(\X(t))$.
These trees have one vanishing edge each (corresponding to $v$ and $v'$ respectively).
There is a unique tree approximating $T(0^\pm)$ (resp. $T'(0^\pm)$)
for any generic perturbation of the configuration $\X(0)$. The only
possible difference in the resulting combinatorial type is the exchange of $p_j(t)$ and $p_{j'}(t)$
on $E(t)$ and $E'(t)$. It does not affect the slopes of the edges and thus
the multiplicity of the curves.

If $h(E(t))\ni p_j(t)$ but $h(E'(t))\cap\X(t)=\emptyset$ then  $v(t),v'(t)$ belong to the same component $T(t)$
of $C\setminus h(t)^{-1}(\X(t))$ and this component has two disjoint vanishing edges.
Let $T\subset C$ be the limit of $T(t)$ when $t\to 0^\pm$.
Suppose the unbounded edge of $T$  belongs to the component of $T\setminus \{v,v'\}$ adjacent to $v'$.

We may treat the vanishing edges one by one.
First we consider the perturbation of the vertex $v$, where the position of
the lines containing the results of perturbation of three out of four adjacent edges (all except for $E'$)
are inductively determined by $\X(t)$ and the slopes of the combinatorial type.
In its turn, the combinatorial type of the perturbation near $v$ is unique as two edges of $C$ adjacent to $v$
emanate in the same direction.

This determines both trivalent vertices that approximate $v$ as well as the line containing $E(t)$.
We proceed with the perturbation of $v'$ in the same way. Once again we get that there is a unique
combinatorial type of $h(t)\in\S(g,\Delta,\X(t))$ approximating $h$ for each generic perturbation $\X(t)$ of $\X(0)$
and it multiplicity does not depend on the choice of perturbation.
\end{proof}


\bibliography{b}
\bibliographystyle{plain}

\end{document}